\documentclass[12pt]{article}
\usepackage{amsmath,amsthm,amssymb,enumitem}
\usepackage{graphicx}

\newcommand{\ignore}[1]{}

\newtheorem{theorem}{Theorem}[section]
\newtheorem{lemma}[theorem]{Lemma}
\newtheorem{corollary}[theorem]{Corollary}

\newcommand{\RR}{\ensuremath{\mathbb R}}
\newcommand{\pts}{\mathcal P}
\newcommand{\ptss}{{\mathcal P}^{2*}}
\newcommand{\ptsss}{{\mathcal P}^{4*}}
\newcommand{\vrt}{S}
\newcommand{\vrts}{\mathcal S}

\newcommand{\surfs}{\mathcal S}
\newcommand{\surfsc}{\mathcal{\overline{S}}}

\newcommand{\bi}[1]{\mathcal{B}(#1)}
\newcommand{\bis}{\mathcal B}
\newcommand{\bism}{\mathcal B^*}
\newcommand{\en}{\mathcal E}
\newcommand{\refl}{\textbf{R}}
\newcommand{\eps}{\varepsilon}
\def\cp{{C_\text{part}}}
\def\cells{{C_\text{cells}}}
\def\inter{{C_\text{inter}}}
\def\chol{{C_\text{H\"old}}}
\def\comps{{C_\text{comps}}}
\def\minus{\backslash}

\begin{document}
\pagenumbering{arabic}
\title{Bisector energy and few distinct distances\thanks{Part of this research was performed while the authors were visiting the Institute for Pure and Applied Mathematics (IPAM) in Los Angeles, which is supported by the National Science Foundation. Work on this paper by Frank de Zeeuw was partially supported by Swiss National Science Foundation Grants 200020-144531 and 200021-137574.
Work on this paper by Ben Lund was supported by NSF grant CCF-1350572.}}
\author{
Ben Lund\thanks{
Department of Computer Science, Rutgers University, 110
Frelinghuysen Road, Piscataway, New Jersey 08854-8004.
{\sl lund.ben@gmail.com}}
\and
Adam Sheffer\thanks{%
Department of Mathematics,
California Institute of Technology,
1200 East California Blvd
Pasadena, CA 91125.
{\sl adamsh@caltech.edu}}
\and
Frank de Zeeuw\thanks{
EPFL, Lausanne, Switzerland.
{\sl fdezeeuw@gmail.com}}
}

\maketitle

\begin{abstract}
We introduce the \emph{bisector energy} of an $n$-point set $\pts$ in $\RR^2$, defined as
\[\en(\pts) =
\left|\left\{(a,b,c,d)\in\pts^4 \mid
\text{$a,b$ have the same perpendicular bisector as $c,d$}
\right\}\right|.\]
If no line or circle contains $M(n)$ points of $\pts$, then we prove that for any $\eps>0$
\[\en(\pts) = O\left(M(n)^{\frac{2}{5}}n^{\frac{12}{5}+\eps} + M(n)n^2\right). \]
We also derive the lower bound $\en(\pts)=\Omega(M(n)n^2)$,
which matches our upper bound when $M(n)$ is large.

We use our upper bound on $\en(\pts)$ to obtain two rather different results:\\[2mm]
(i) If $\pts$ determines $O(n/\sqrt{\log n})$ distinct distances, then for any $0<\alpha\le 1/4$, either there exists a line or circle that contains $n^\alpha$ points of $\pts$, or there exist $\Omega(n^{8/5-12\alpha/5-\eps})$ distinct lines that contain $\Omega(\sqrt{\log n})$ points of $\pts$.
This result provides new information on a conjecture of Erd\H os \cite{Erd86} regarding the structure of point sets with few distinct distances.
\\[2mm]
(ii) If no line or circle contains $M(n)$ points of $\pts$,
the number of distinct perpendicular bisectors determined by $\pts$ is
$\Omega\left(\min\left\{M(n)^{-2/5}n^{8/5-\eps},
M(n)^{-1} n^2\right\}\right)$. This appears to be the first higher-dimensional example in a framework for studying the expansion properties of polynomials and rational functions over $\RR$,
initiated by Elekes and R\'onyai \cite{ER00}.
\end{abstract}

\section{Introduction}\label{sec:intro}

Guth and Katz \cite{GK11} proved that every set of $n$ points in $\RR^2$ determines $\Omega(n/ \log n)$ distinct distances. This almost completely settled a conjecture of Erd\H os \cite{Erd46}, who observed that the $\sqrt{n}\times \sqrt{n}$ integer lattice determines $\Theta(n/\sqrt{\log n})$ distances, and conjectured that every set of $n$ points determines at least this number of  distances. Beyond the remaining $\sqrt{\log n}$ gap, this leaves open the question of which point sets determine few distances.
Erd\H os \cite{Erd86} asked whether every set that determines $O(n/\sqrt{\log n})$ distances ``has lattice structure''.
He then wrote: \emph{``The first step would be to decide if there always is a line which contains $cn^{1/2}$ of the points (and in fact $n^{\eps}$ would already be interesting)}.''

Embarrassingly, almost three decades later the bound $n^{\eps}$ seems as distant as it ever was.
The following bound is a consequence of an argument of Szemer\'edi, presented by Erd\H os \cite{Erd75}.

\begin{theorem}[Szemer\'edi]\label{th:Szem}
If a set $\pts$ of $n$ points in $\RR^2$ determines $O(n/\sqrt{\log n})$ distances, then there exists a line containing $\Omega(\sqrt{\log n})$ points of $\pts$.
\end{theorem}

Recently, it was noticed that this bound can be slightly improved to $\Omega(\log n)$ points on a line (see \cite{Sh14b}).
Assuming that no line contains an asymptotically larger number of points, one can prove the existence of $\Omega(n/\log n)$ distinct lines that contain $\Omega(\log n)$ points of $\pts$.
By inspecting Szemer\'edi's proof, it is also apparent that these lines are \emph{perpendicular bisectors} of pairs of points of $\pts$.

This problem was recently approached from the other direction in \cite{PZ13,RRS14,SZZ14}. Combining the results of these three papers implies the following.
If an $n$-point set $\pts\subset\RR^2$ determines $o(n)$ distances, then no line contains $\Omega(n^{43/52+\eps})$ points of $\pts$,
no circle contains $\Omega(n^{5/6})$ points, and no other constant-degree irreducible algebraic curve contains $\Omega(n^{3/4})$ points.

In the current paper we study a different aspect of sets with few distinct distances. Our main tool is a bound on the \emph{bisector energy} of the point set (see below for a formal definition).
Using this tool, we prove that if a point set $\pts$ determines $O(n/\sqrt{\log n})$ distinct distances, then either there exists a line or a circle with many points of $\pts$, or the number of lines containing $\Omega(\sqrt{\log n})$ must be significantly larger than implied by Theorem \ref{th:Szem}.
As another application of bisector energy, we prove that if no line or circle contains many points of a point set $\pts$, then $\pts$ determines a large number of distinct perpendicular bisectors.
We will provide more background to both results after we have properly stated them.

\section{Results} \label{sec:results}

\paragraph{Bisector energy.}
Given two distinct points $a,b\in \RR^2$, we denote by $\bi{a,b}$ their \emph{perpendicular bisector} (i.e., the line consisting of all points that are equidistant from $a$ and $b$); for brevity, we usually refer to it as the \emph{bisector} of $a$ and $b$.
We define the \emph{bisector energy} of $\pts$ as
\[\en(\pts) =
\left|\left\{(a,b,c,d)\in\pts^4 \mid
\ a\neq b, \ c\neq d,\ \text{and}~ \
\bi{a,b} = \bi{c,d}\right\}\right|.\]
In Section \ref{sec:proofofmain}, we prove the following upper bound on this quantity.

\begin{theorem}\label{th:energy}
For any $n$-point set $\pts\subset \RR^2$, such that no line or circle contains $M(n)$ points of $\pts$, we
have\footnote{Throughout this paper, when we state a bound involving an $\eps$, we mean that this bound holds for every $\eps>0$, with the multiplicative constant of the $O()$-notation depending on $\eps$.}
\[\en(\pts) = O\left(M(n)^{\frac{2}{5}}n^{\frac{12}{5}+\eps} + M(n)n^2\right).\]
\end{theorem}
The bound of Theorem \ref{th:energy} is dominated by its first term
when $M(n)=O(n^{2/3+\eps'})$.
We note that one important ingredient of our proof is the result of Guth and Katz \cite{GK11}; without it, we would obtain a weaker (although nontrivial) bound on the bisector energy (see the remark at the end of Section \ref{sec:applyingbound}).

In parallel to our work, Hanson, Iosevich, Lund, and Roche-Newton \cite{HILR14} derived a variant of Theorem \ref{th:energy} for the case of point sets in $\mathbb{F}_q^2$.

In Section \ref{ssec:lowerbound}, we derive a lower bound for the maximum bisector energy. It shows that Theorem \ref{th:energy} is tight when its second term dominates, i.e., when $M(n)=\Omega(n^{2/3+\eps'})$.

\begin{theorem}\label{th:lowerbound}
For any $n$ and $M(n)$, there exists a set $\pts$ of $n$ points in $\mathbb{R}^2$ such that no line or circle contains $M(n)$ points of $\pts$, and
$\en(\pts) = \Omega\left(M(n)n^2 \right)$.
\end{theorem}
We conjecture that $\en(\pts) = O(M(n)n^2)$  is true for all $M(n)$.

\paragraph{Few distinct distances.}
Our first application of Theorem \ref{th:energy} is to deduce the following theorem.
It follows from the slightly more general Theorem \ref{th:fewdistances2} that we prove in Section \ref{sec:fewdistances}.
\begin{theorem} \label{th:fewdistances}
Let $\pts\subset \RR^2$ be a set of $n$ points that spans $O(n/\sqrt{\log n})$ distinct distances.
Then for any $0<\alpha\le 1/4$, at least one of the following holds.\\
(i) There exists a line or a circle containing $\Omega(n^\alpha)$ points of $\pts$. \\
(ii) There are
$\Omega(n^{\frac{8}{5}-\frac{12\alpha}{5}-\eps})$
lines that contain $\Omega(\sqrt{\log n})$ points of $\pts$.
\end{theorem}
If our conjecture that $\en(\pts)=O(M(n)n^2)$ were proved, alternative $(ii)$ in the conclusion of Theorem \ref{th:fewdistances} would improve to give $\Omega(n^{2 - 3\alpha}\log(n))$ lines that each contain $\Omega(\sqrt{\log n})$ points of $\pts$.

We believe that Theorem \ref{th:fewdistances} is a step towards Erd\H os's lattice conjecture.
We mention several recent results and conjectures that together paint an interesting picture.

Green and Tao \cite{GT13} proved that, given an $n$-point set in $\RR^2$ such that more than $n^2/6-O(n)$ lines contain at least three of the points, most of the points must lie on a cubic curve (an algebraic curve of degree at most three).
Elekes and Szab\'o \cite{ES13} stated the stronger conjecture that if an $n$-point set determines $\Omega(n^2)$ collinear triples, then many of the points lie on a cubic curve; 
unfortunately, at this point it is not even known whether there must be a cubic that contains \emph{ten} points of the set.
Erd\H os and Purdy \cite{EP76} conjectured that if $n$ points determine $\Omega(n^2)$ collinear \emph{quadruples}, then there must be five points on a line.
If the point set is already known to lie on a low-degree algebraic curve,
then both conjectures hold \cite{ES13, RSZ14}. 
On the other hand, Solymosi and Stojakovi\'c \cite{SoSt13} proved that for any constant $k$, there can be $\Omega(n^{2-\eps})$ lines with exactly $k$ points, but no line with $k+1$ points.

The philosophy of these statements is that if there are many lines containing many points, then the points must lie on some low-degree algebraic curve.
Our result shows that for an $n$-point set with few distinct distances,
 either there is a line or circle with very many points,
or else there are many lines with many points.
In particular, in the second case there would be many collinear triples (although not quite as many as $\Omega(n^2)$),
and many lines with very many (more than a constant) points.
This suggests that few distinct distances should imply some algebraic structure.
Let us pose a specific question: Is there a $0<\beta<1$ such that if $n$ points determine $\Omega(n^{1+\beta})$ lines with $\Omega(\sqrt{\log n})$ points, then many of the points must lie on a low-degree algebraic curve?

\paragraph{Distinct bisectors.}
Let $\bis(\pts)$ be the set of those lines that are distinct perpendicular bisectors of $\pts$.
Since any point of $\pts$ determines $n-1$ distinct bisectors with the other points of $\pts$, we have a trivial lower bound $|\bis(\pts)|\ge n-1$.
If $\pts$ is a set of equally spaced points on a circle, then $|\bis(\pts)|= n$. Similarly, if $\pts$ is a set of $n$ equally spaced points on a line, then $|\bis(\pts)|= 2n-3$.
As we now show, forbidding many points on a line or circle forces $|\bis(\pts)|$ to be significantly larger.

\begin{theorem}\label{th:distinctbisectors}
If an $n$-point set $\pts\subset \RR^2$ has no $M(n)$ points on a line or circle, then
\[|\bis(\pts)| =
\Omega\left(\min\left\{M(n)^{-\frac{2}{5}}n^{\frac{8}{5}-\eps},
M(n)^{-1} n^2\right\}\right).\]
\end{theorem}
\begin{proof} For any line $\ell\subset \RR^2$, set $E_\ell = \{(a,b)\in \pts^2\mid a\neq b, ~\bi{a,b} = \ell\}$. By the Cauchy-Schwarz inequality, we have
\[|\en(\pts)|  = \sum_{l\in \bis(\pts)} |E_l|^2 \geq \frac{1}{|\bis(\pts)|}\left(\sum_{l\in \bis(\pts)} |E_l|\right)^2 = \Omega\left(\frac{n^4}{|\bis(\pts)|}\right)\]
Combining this with the bound of Theorem \ref{th:energy} immediately implies the theorem.
\end{proof}

We are not aware of any previous results concerning the minimum number of distinct bisectors.

Theorem \ref{th:distinctbisectors} can be viewed as a new step in a series of results initiated by Elekes and R\'onyai \cite{ER00}, studying the expansion properties of polynomials and rational functions over $\RR$. For instance, in \cite{RSS14} it is proved that a polynomial function $F:\RR\times \RR\to \RR$ takes $\Omega(n^{4/3})$ values on any set in $\RR$ of size $n$, unless $F(x,y) = G(H(x)+K(y))$ or $F(x,y) = G(H(x)K(y))$ with polynomials $G,H,K$.
Elekes and Szab\'o \cite{ES12} derived the following higher-dimensional generalization (rephrased for our convenience, and omitting some details).
If $F:\RR^D\times \RR^D\to \RR^D$ is a rational function that is not of a special form, and $A,B \subset \RR^D$ are two $n$-point sets such that no low-degree proper subvariety of $\RR^D$ contains many points of $A$ or $B$, then $F$ takes $\Omega(n^{1+\eps})$ values on $A\times B$.
However, this last condition is hard to use, and no concrete case with $D>1$ is known.

Theorem \ref{th:distinctbisectors}
is such a concrete case, for the function $\bis$.
If we view a line $y=sx+t$ as a point $(s,t)\in \RR^2$, then (see the proof of Lemma \ref{lem:surfaces})
\[\bis(a_x,a_y,b_x,b_y)
= \left(-\frac{a_x-b_x}{a_y-b_y}, \frac{(a_x^2+a_y^2) - (b_x^2+b_y^2)}{2(a_y-b_y)}\right)\]
is a rational function $\RR^2\times \RR^2\to \RR^2$.
Then Theorem \ref{th:distinctbisectors} says that $\bis$ takes many distinct values on any $n$-point set with few points on a line or circle.
So we have replaced the broad condition of \cite{ES12} that not too many points lie on a low-degree curve, with the very specific condition that not too many points lie on a line or circle.
Moreover, the exponent in our lower bound is considerably better than that of \cite{ES12}.

\paragraph{An incidence bound.}
To prove Theorem \ref{th:energy}, we use the incidence bound below.
It is a refined version of a theorem from Fox et al.~\cite{FPSSZ14}, with explicit dependence on the parameter $t$, which we allow to depend on $m$ and $n$.
We reproduce the proof in Section \ref{sec:incidencebound} to determine this dependence.
Given a set $\pts \subset \RR^d$ of points and a set $\vrts \subset \RR^d$ of varieties, the \emph{incidence graph} is a bipartite graph with vertex sets $\pts$ and $\vrts$, such that $(p,\vrt)\in \pts \times \vrts$ is an edge in the graph if the point $p$ is incident to the variety $\vrt$.
We write $I(\pts,\vrts)$ for the number of edges of this graph, or in other words, for the number of \emph{incidences} between $\pts$ and $\vrts$.
We denote the complete bipartite graph on $s$ and $t$ vertices by $K_{s,t}$.

\begin{theorem}\label{th:incidencebound}
Let $\vrts$ be a set of $n$ constant-degree varieties and let $\pts$ be a set of $m$ points, both in $\RR^d$, such that the incidence graph of $\pts\times\vrts$ contains no copy of $K_{s,t}$ (where $s$ is a constant, but $t$ may depend on $m,n$).
Moreover, let $\pts\subset V$, where $V$ is an irreducible constant-degree variety of dimension $e$.
Then
\[I(\pts,\vrts) = O\left(m^{\frac{s(e-1)}{es-1}+\eps}n^{\frac{e(s-1)}{es-1}}t^{\frac{e-1}{es-1}} +tm+n\right).\]
\end{theorem}


\section{Proof of Theorem \ref{th:energy}}\label{sec:proofofmain}
In this section we prove Theorem \ref{th:energy} by relating the bisector energy to an incidence problem between points and algebraic surfaces in $\RR^4$.
In Section \ref{sec:surfaces} we define the surfaces, in Section \ref{sec:intersections} we analyze their intersection properties,
and in Section \ref{sec:applyingbound} we apply the incidence bound of Theorem \ref{th:incidencebound} to prove Theorem \ref{th:energy}. Finally, in Section \ref{ssec:lowerbound} we derive Theorem \ref{th:lowerbound}, which provides a lower bound for Theorem \ref{th:energy}.

Throughout this section we assume that we have rotated $\pts$ so that no two points have the same $x$- or $y$-coordinate;
in particular, we assume that no perpendicular bisector is horizontal or vertical.

\subsection{Bisector surfaces}\label{sec:surfaces}
Recall that in Theorem \ref{th:energy} we consider an $n$-point set $\pts\subset \RR^2$.
We define
\[\ptss = \{(a,c)\in \pts^2\mid a\neq c\},\]
and similarly
\[ \ptsss = \{(a,b,c,d)\in \pts^4\mid a\neq c, b\neq d\}.\]
Recall also that for a pair $(a,b)\in \ptss$, we denote by $\bi{a,b}$ the perpendicular bisector of $a$ and $b$.
We define the \emph{bisector surface} of a pair $(a,c)\in \ptss$ as
\[S_{ac} = \{(b,d)\in \RR^4\mid (a,b,c,d)\in\ptsss,~\bi{a,b} = \bi{c,d}\},\]
and we set $\surfs = \{ S_{ac}\mid (a,c)\in \ptss\}$.
The surface $S_{ac}$ is not an algebraic variety (so we are using the word ``surface'' loosely), but the lemma below shows that $S_{ac}$ is ``close to'' a variety $\overline{S}_{ac}$.
That $S_{ac}$ is contained in a constant-degree variety of the same dimension is no surprise (one can take the \emph{Zariski closure}), but we need to analyze this variety in detail to establish the exact relationship.

We will work mostly with the surface $S_{ac}$ in the rest of this proof, rather than with the variety $\overline{S}_{ac}$, because its definition is easier to handle.
Then, when we apply our incidence bound, which holds only for varieties, we will switch to $\overline{S}_{ac}$. Fortunately, the lemma shows that  this makes no difference in terms of the incidence graphs.

\begin{lemma}\label{lem:surfaces}
For distinct $a,c\in \pts$, there exists a two-dimensional constant-degree algebraic variety $\overline{S}_{ac}$ such that $S_{ac}\subset \overline{S}_{ac}$.
Moreover, if $(b,d)\in (\overline{S}_{ac}\backslash S_{ac})\cap \ptss$,
then $(a,b,c,d)\not\in \ptsss$.
\end{lemma}
\begin{proof}
Consider a point $(b,d)\in S_{ac}$ with $a\neq b$ and $c\neq d$.
Write the equation defining the perpendicular bisector $\bi{a,b}=\bi{c,d}$ as $y=sx+t$. The slope $s$ satisfies
\begin{equation} \label{eq:cond1}
s = -\frac{a_x-b_x}{a_y-b_y} = -\frac{c_x-d_x}{c_y-d_y}.
\end{equation}
By definition $\bi{a,b}$ passes through the midpoint $((a_x+b_x)/2$, $(a_y+b_y)/2)$ of $a$ and $b$, as well as through the midpoint $\left((c_x+d_x)/2,(c_y+d_y)/2\right)$ of $c$ and $d$. We thus have
\[  \frac{a_y+b_y}{2} - s\frac{a_x+b_x}{2} = t =  \frac{c_y+d_y}{2} - s\frac{c_x+d_x}{2}.\]
By replacing $s$ with both of the other expressions in \eqref{eq:cond1} and rearranging, we obtain
\begin{align}
(a_y-b_y)(c_x^2+c_y^2-d_x^2-d_y^2)
= (c_y-d_y)(a_x^2+a_y^2-b_x^2-b_y^2).
\label{eq:cond2}
\end{align}
 From \eqref{eq:cond1} and \eqref{eq:cond2} we see that $(b,d)=(x_1,x_2,x_3,x_4)$ satisfies
\begin{align*}
f_{ac}(x_1,x_2,x_3,x_4) &= (a_x-x_1)(c_y-x_4) - (a_y-x_2)(c_x-x_3) = 0, \\[2mm]
g_{ac}(x_1,x_2,x_3,x_4) &= (a_y-x_2)(c_x^2+c_y^2-x_3^2-x_4^2)
- (c_y-x_4)(a_x^2+a_y^2-x_1^2-x_2^2)=0.
\end{align*}
Since any point $(b,d)\in S_{ac}$ satisfies these two equations, we have
\[S_{ac}\subset Z(f_{ac},g_{ac}) = \overline{S}_{ac}.\]

By reexamining the above analysis, we see that if a point $(b,d)\in \overline{S}_{ac}\cap \ptss$ is not in $S_{ac}$,
we must have $a_y=b_y$ or $c_y=d_y$, since then \eqref{eq:cond1} is not well defined.
By the assumption that no two points of $\pts$ have the same $y$-coordinate,
this implies $a = b$ or $c = d$, so $(a,b,c,d)\not\in \ptsss$.

It remains to prove that $\overline{S}_{ac}$ is a constant-degree two-dimensional variety. The constant degree is immediate from $f_{ac}$ and $g_{ac}$ being polynomials of degree at most three.
As just observed, points in $(b,d)\in \overline{S}_{ac} \minus S_{ac}$ satisfy   $a_y=b_y$ or $c_y=d_y$. If $a_y=b_y$, then for $f_{ac}(b,d)=g_{ac}(b,d)=0$ to hold, either $a_x=b_x$ or $c_y=d_y$. Similarly, If $c_y=d_y$, then either $c_x=d_x$ or $a_y=b_y$.
We see that in each case we get two independent linear equations, which define a plane,
so $\overline{S}_{ac} \minus S_{ac}$ is the union of three two-dimensional planes. Thus, it suffices to prove that $S_{ac}$ is two-dimensional.
For this, we simply show that for any valid value of $b$ there is at most one valid value of $d$.
Let $C_{ac}\subset \RR^2$ denote the circle that is centered at $c$ and incident to $a$.
It is impossible for $b$ to lie on $C_{ac}$, since this would imply that the bisector $\bi{a,b}$ contains $c$, and thus that $\bi{a,b}\neq\bi{c,d}$.
For any choice of $b\notin C_{ac}$, the bisector $\bi{a,b}$ is well-defined and is not incident to $c$, so there is a unique $d\in \RR^2$ with $\bi{a,b}=\bi{c,d}$ (i.e., so that $(b,d)\in S_{ac}$).
\end{proof}

\subsection{Intersections of bisector surfaces}\label{sec:intersections}

We denote by $\refl_{ab}$ the reflection of $\RR^2$ across the line $\bi{a,b}$.
Observe that if $\bi{a,b} = \bi{c,d}$, then $\refl_{ab} = \refl_{cd}$, and this reflection maps both $a$ to $b$ and $c$ to $d$; this in turn implies that $|ac| = |bd|$.
That is, $(b,d)\in S_{ac}$ implies $|ab|=|cd|$.
It follows that if $|ac| = \delta$, then the open surface $S_{ac}$ is contained in the hypersurface
\[H_\delta = \{(b,d)\in \RR^4 \mid |bd| = \delta\}.\]
We can thus partition $\surfs$ into classes corresponding to the distances $\delta$ that are determined by pairs of points of $\pts$. Each class consists of the surfaces $S_{ac}$ with $|ac|=\delta$, all of which are fully contained in $H_\delta$.

We now study the intersection of the surfaces contained in a common hypersurface $H_\delta$.

\begin{lemma}\label{lem:intersection}
Let $(a,c)\neq (a',c')$ and $|ac| = |a'c'| = \delta \neq0$. Then there exist two curves $C_1,C_2\subset \RR^2$, which are either two concentric circles or two parallel lines,
such that $a,a'\in C_1$, $c,c'\in C_2$, and $S_{ac}\cap S_{a'c'}$ is contained in the set
\[ H_\delta\cap(C_1\times C_2) = \{(b,d)\in \RR^4\mid b\in C_1,d\in C_2,|bd| = \delta\}.\]
\end{lemma}
\begin{proof}
We split the analysis into three cases: (i) $|\bi{a,a'}\cap \bi{c,c'}|=1$, (ii) $\bi{a,a'} = \bi{c,c'}$, and (iii) $\bi{a,a'}\cap \bi{c,c'} =\emptyset$. The three cases are depicted in Figure \ref{fi:cases}.

\begin{figure}[ht]
\centerline{\includegraphics[width=0.7\textwidth]{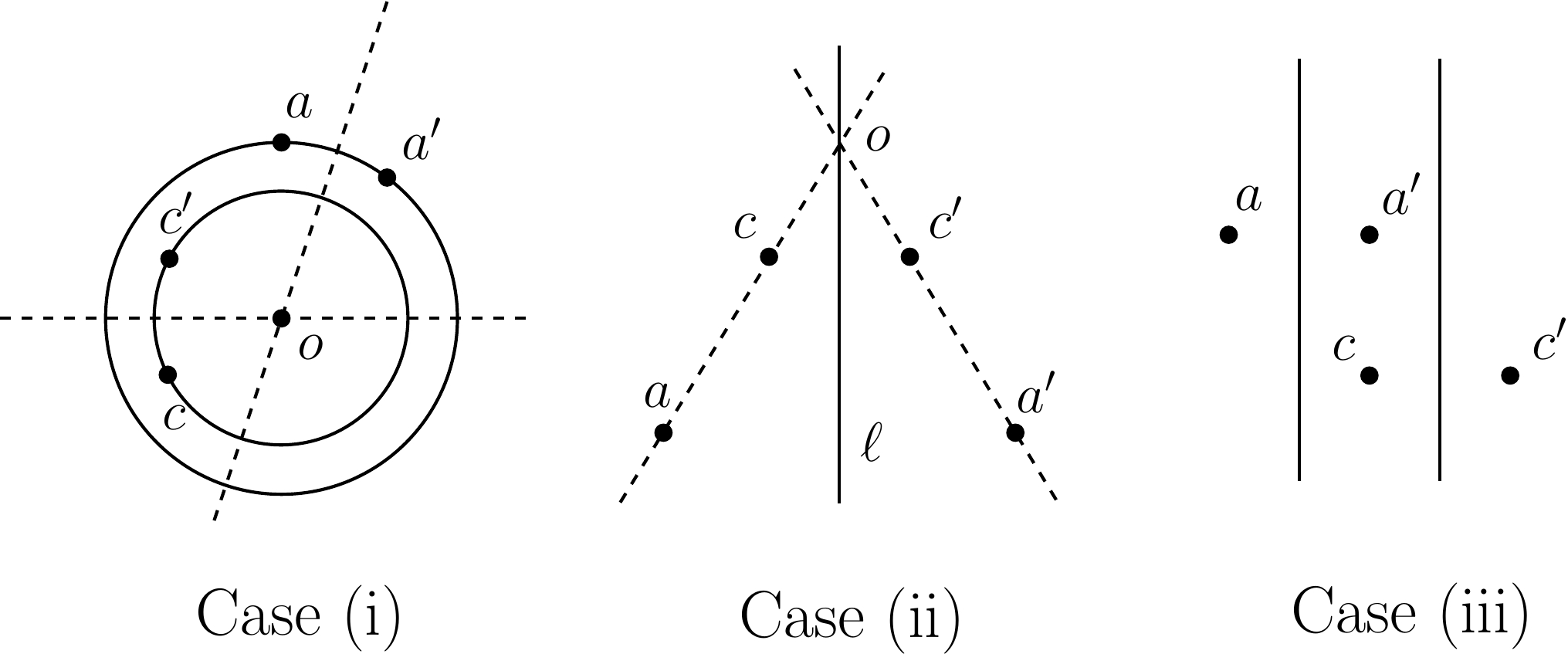}}
\caption{\small \sf The three cases in the analysis of Lemma \ref{lem:intersection}.}
\label{fi:cases}
\end{figure}

{\bf Case (i).} Let $o = \bi{a,a'}\cap \bi{c,c'}$.
Then there exist two (not necessarily distinct) circles $C_1,C_2$ around $o$ such that $a,a'\in C_1$ and $c,c'\in C_2$ .
If $(b,d)\in S_{ac}\cap S_{a'c'}$, then the reflection $\refl_{ab}$ takes $a$ to $b$ and $c$ to $d$, and similarly, $\refl_{a'b}$ takes $b$ to $a'$ and $d$ to $c'$.
We set $\textbf{T} = \refl_{a'b}\circ \refl_{ab}$, and notice that this is a rotation whose center $o^*$ is the intersection point of  $\bi{a,b}=\bi{c,d}$ and $\bi{a',b}=\bi{c',d}$.
Note that $\textbf{T}(a) = a'$ and $\textbf{T}(c) = c'$, so $o^*$ lies on both $\bi{a,a'}$ and $\bi{c,c'}$. Since $o=\bi{a,a'}\cap \bi{c,c'}$, we obtain that $o=o^*$.
Since $\bi{a,b}$ passes through $o$, we have that $b$ is incident to $C_1$. Similarly, since $\bi{c,d}$ passes through $o$, we have that $d$ is incident to $C_2$.
This implies that $(b,d)$ lies in $H_\delta\cap (C_1\times C_2)$.

{\bf Case (ii).}
Let $\ell$ be the line $\bi{a,a'} = \bi{c,c'}$.
The line segment $ac$ is a reflection across $\ell$ of the line segment $a'c'$. Thus, the intersection point $o$ of the lines that contains these two segments is incident to $\ell$. Let $C_1$ be the circle centered at $o$ that contains $a$ and $a'$, and let $C_2$ be the circle centered at $o$ that contains $c$ and $c'$.
With this definition of $o$, $C_1$, and $C_2$, we can repeat the analysis of case (i), obtaining the same conclusion.

{\bf Case (iii).} In this case $\bi{a,a'}$ and $\bi{c,c'}$ are parallel.
The analysis of this case is similar to that in case (i),
but with lines instead of circles.

Let $C_1$ be the line that is incident to $a$ and $a'$, and let $C_2$ be the line that is incident to $c$ and $c'$.
If $(b,d)\in S_{ac}\cap S_{a'c'}$, then, as before, $\refl_{ab}$ takes $a$ to $b$ and $c$ to $d$, and $\refl_{a'b}$ takes $b$ to $a'$ and $d$ to $c'$.
Since $\bi{a',b}$ and $\bi{a,b}$ are parallel, we have that $\textbf{T} = \refl_{a'b}\circ \refl_{ab}$ is a translation in the direction orthogonal to these two lines. This implies that $b\in C_1$ and $d\in C_2$, which completes the analysis of this case.
\end{proof}

In Section \ref{sec:applyingbound}, we will apply the incidence bound of Theorem \ref{th:incidencebound} to the point set $\ptss=\{(b,d)\in \pts^2\mid b\neq d\}$ and the set of surfaces $\surfs$.
For this we need to show that the incidence graph contains no complete bipartite graph $K_{2,M}$;
that is, that for any two points of $\ptss$ (where $\ptss$ is considered as a point set in $\RR^4$) there is a bounded number of surfaces of $\surfs$ that contain both points.
In the following lemma we prove the more general statement that the incidence graph contains no copy of $K_{2,M}$ and no copy of $K_{M,2}$.
Note that this is the only point in the proof of Theorem \ref{th:energy} where we use the condition that no $M$ points are on a line or circle.

\begin{corollary}\label{cor:nobipartite}
If no line or circle contains $M$ points of $\pts$,
then the incidence graph of $\ptss$ and $\surfs$ does not contain a copy of $K_{2,M}$ or $K_{M,2}$.
\end{corollary}
\begin{proof}
Consider two distinct surfaces $S_{ac},S_{a'c'}\in\surfs$ with $|ac|=|a'c'|=\delta$. Lemma \ref{lem:intersection} implies that there exist two lines or circles $C_1,C_2$ such that
$(b,d)\in S_{ac}\cap S_{a'c'}$ only if $b\in C_1$ and $d\in C_2$.
Since no line or circle contains $M$ points of $\pts$, we have $|C_1\cap \ptss|< M$.
Given $b\in (C_1\cap \pts)\backslash \{a\}$,
there is at most one $d\in \pts$ such that $\bi{a,b} = \bi{c,d}$, and thus at most one point $(b,d)\in S_{ac}$. (Notice that no points of the form $(a,d)\in\ptss$ are in $S_{ac}$.)
Thus
\[|(S_{ac}\cap S_{a'c'})\cap \ptss| < M.\]
That is, the incidence graph contains no copy of $K_{M,2}$.

We now define ``dual'' surfaces
\[S_{bd}^* = \{(a,c)\in \RR^4\mid a\neq b, c\neq d, \bi{a,b} = \bi{c,d}\},\]
and set $\mathcal{S}^* = \{S_{bd}^*\mid (b,d)\in \ptss\}$.
By a symmetric argument, we get
\[|(S^*_{bd}\cap S^*_{b'd'})\cap \ptss| < M\]
for all $(b,d)\neq (b',d')$.
Observe that $(a,c)\in S^*_{bd}$ if and only if $(b,d)\in S_{ac}$.
Hence, having fewer than $M$ points $(a,c)\in (S^*_{bd}\cap S^*_{b'd'})\cap \ptss$ is equivalent to having fewer than $M$ surfaces $S_{ac}$ that are incident to both $(b,d)$ and $(b',d')$.
That is, the incidence graph contains no copy of $K_{2,M}$.
\end{proof}


\subsection{Applying the incidence bound}\label{sec:applyingbound}
We set
\[ Q = \{ (a,b,c,d) \in \ptsss \mid \ \bi{a,b} = \bi{c,d}\}, \]
and note that $|Q|+\binom{n}{2} = \en(\pts)$,
where the term $\binom{n}{2}$ accounts for the quadruples of the form $(a,b,a,b)$.
As we saw in Section \ref{sec:intersections}, every quadruple $(a,b,c,d)\in Q$ satisfies $|ac|=|bd|$.

Let $\delta_1,\ldots, \delta_D$ denote the distinct distances that are determined by pairs of distinct points in $\pts$.
We partition $\ptss$ into the disjoint subsets $\Pi_1,\ldots,\Pi_D$, where
\[\Pi_i = \{(u,v)\in\ptss\mid |uv|=\delta_i\}.\]
We also partition $\surfs$ into disjoint subsets $\surfs_1,\ldots,\surfs_D$,
defined by
\[\surfs_i = \{S_{ac}\in \surfs\mid |ac| = \delta_i\}.\]
%
Let $m_i$ be the number of $(a,c)\in\ptss$ such that $|ac| =\delta_i$.
Note that $|\Pi_i| = |\surfs_i| = m_i$ and
\[\sum m_i = n(n-1).\]

A quadruple $(a,b,c,d)\in\ptsss$ is in $Q$ if and only if the point $(b,d)$ is incident to $S_{ac}$.
Moreover, there exists a unique $1\le i\le D$ such that $(b,d)\in \Pi_i$ and $S_{ac}\in \surfs_i$. Therefore, it suffices to study each $\Pi_i$ and $\surfs_i$ separately. That is, we have
\[ |Q| =\sum_{i=1}^D I(\Pi_i,\surfs_i).\]

We apply our incidence bound to $\surfs_i$, or rather, to the corresponding set of varieties $\surfsc_i=\{\overline{S}_{ac}\mid S_{ac}\in \surfs_i\}$.
By Lemma \ref{lem:surfaces}, the incidence graph of $\Pi_i$ with $\surfsc_i$ is the same as with $\surfs_i$,
hence also does not contain any copy of $K_{2,M}$ by Corollary \ref{cor:nobipartite}.
Observe that $\Pi_i\subset H_{\delta_i}$.
The hypersurface $H_{\delta_i}$ is irreducible, three-dimensional, and of a constant degree, since it is defined by the irreducible polynomial $(x_1-x_3)^2+(x_2-x_4)^2 - \delta_i$.
Thus we can apply Theorem \ref{th:incidencebound} to each $I(\Pi_i, \surfsc_i)$,
with $m=n=m_i$, $V = H_{\delta_i} $, $d=4$, $e=3$, $s=2$, and $t=M$.
This implies that
\begin{equation} \label{eq:Inc-i}
I(\Pi_i,\surfs_i) = I(\Pi_i,\surfsc_i) = O\left(M^{\frac{2}{5}} m_i^{\frac{7}{5}+\eps} + Mm_i \right).
\end{equation}

Let $J$ be the set of indices $1\le j\le D$ for which the bound in \eqref{eq:Inc-i} is dominated by the term $M^{\frac{2}{5}} m_j^{\frac{7}{5}+\eps}$. By recalling that $\sum_{j=1}^D m_j = n(n-1)$, we get
\[\sum_{j\not\in J}I(\Pi_j,\surfs_j) = O\left(Mn^2\right).\]
Next we consider $\sum_{j\in J}I(\Pi_j,\surfs_j) = O(\sum_{j\in J} M^{2/5} m_j^{7/5+\eps} )$.
By \cite[Proposition 2.2]{GK11}, we have
\[\sum m_j^2 = O(n^3\log n).\]
This implies that the number of $m_j$ for which $m_j\geq x$ is $O(n^3\log n/x^2)$. By using a dyadic decomposition, we obtain
\begin{align*}
M^{-2/5}n^{-\eps}\sum_{j\in J}I(\Pi_j,\surfs_j) &= O\left( \sum_{m_j \le \Delta} m_j^{7/5} + \sum_{k\ge 1} \sum_{2^{k-1}\Delta<m_j\le 2^{k}\Delta}  m_j^{7/5}\right)  \\
&= O\left(  \Delta^{7/5}\cdot \frac{n^2}{\Delta} + \sum_{k\ge 1} (2^k\Delta)^{\frac{7}{5}} \cdot \frac{n^3\log n}{(2^k\Delta)^2}\right)  \\
&= O\left( \Delta^{2/5}n^2 +  \frac{n^3\log n}{\Delta^{3/5}} \right).
\end{align*}
By setting $\Delta = n\log n$, we have
\[ \sum_{j\in J}I(\Pi_j,\surfs_j)
= O\left(M^{\frac{2}{5}}n^{\frac{12}{5}+\eps} \log^{\frac{2}{5}} n  \right)
= O\left(M^{\frac{2}{5}}n^{\frac{12}{5}+\eps'}  \right).
\]

In conclusion,
\[\en(\pts) \leq  |Q| +n^2
=\sum_{j\in J}I(\Pi_j,\surfs_j)+\sum_{j\not\in J}I(\Pi_j,\surfs_j) +n^2
= O\left(M^{\frac{2}{5}}n^{\frac{12}{5}+\eps'}+Mn^2\right),\]
which completes the proof of Theorem \ref{th:energy}.

\paragraph{Remark about the incidence bound.}
Instead of partitioning the problem into $D$ separate incidence problems,
one can apply an incidence bound directly to the point set $\ptss$ and the surface set $\surfs$.
Roughly speaking, the best known bounds for incidences with two-dimensional surfaces in $\RR^4$, whose incidence graph contains no $K_{2,M}$,
are of the form $|\ptss|^{2/3}|\surfs|^{2/3}$.
Relying on such an incidence bound (and not using the estimate from \cite{GK11}) would yield a bound $|Q| = O(M^{1/3}n^{8/3} +Mn^2)=O(M^{1/3}n^{8/3})$, which is nontrivial but weaker than our bound.




\subsection{A lower bound for $\en(\pts)$}\label{ssec:lowerbound}

In this section we prove Theorem \ref{th:lowerbound}. In particular, for any $n$ and $M(n)\ge 32$, we show that there exists a set $\pts$ of $n$ points in $\mathbb{R}^2$ such that any line or circle contains at most $M(n)$ points of $\pts$, and $\en(\pts) = \Omega\left(M(n)n^2 \right)$.
Note that we can suppose $M(n) \ge 32$ without loss of generality since, if $M(n) < 32$, an arbitrary point set has $\en(\pts) = \Omega\left(n^2\right) = \Omega \left (M(n)n^2 \right)$.

For simplicity, we assume that $M(n)$ is a multiple of 8, and that $n$ is divisible by $M(n)$. It is straightforward to extend the following construction to values that do not satisfy these conditions.

\begin{figure}[ht]
\centerline{\includegraphics[width=0.8\textwidth]{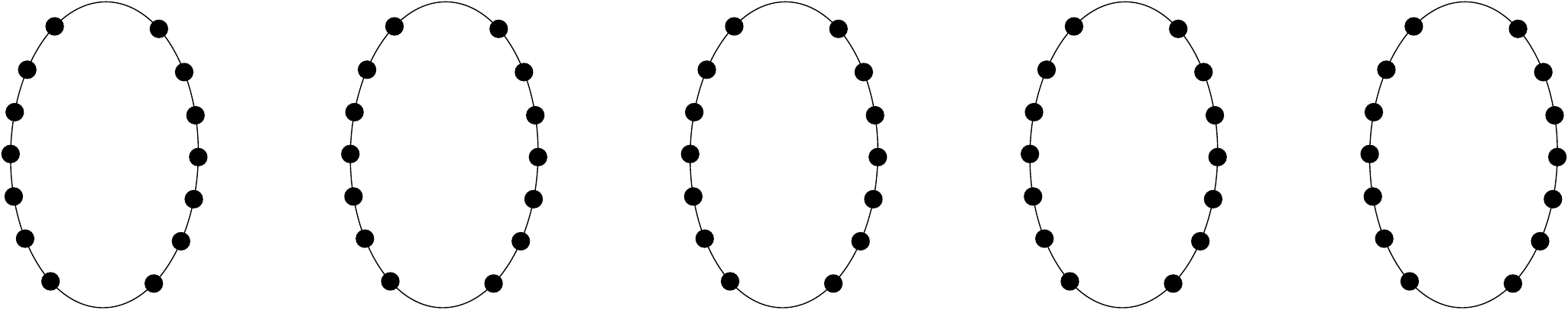}}
\caption{\small \sf The lower bound construction.}
\label{fi:ellipses}
\end{figure}

Let $C$ be an axis-parallel ellipse that is centered at the origin, has a major axis of length $2$ that is parallel to the $y$-axis, and a minor axis of length $1$ that is parallel to the $x$-axis.
Let $\pts^{+}$ be an arbitrary set of $4n/M(n)$ points on $C$, each having a strictly positive $x$-coordinate.
Let $\pts^{-}$ be the reflection of $\pts^+$ over the $y$-axis, and set $\pts' = \pts^+ \cup \pts^-$.
We denote by $\pts'_j$ the translate of $\pts'$ by $(4j,0)$.
Finally, we take $\pts = \pts'_0 \cup \pts'_1 \cup \cdots \cup\pts'_{M(n)/8 -1}$. An example is depicted in Figure \ref{fi:ellipses}.

Note that $\pts$ lies on the union of $M(n)/8$ ellipses. Since a line can intersect an ellipse in at most two points, and a circle can intersect an ellipse in at most four points, we indeed have that a line or circle contains at most $M(n)$ points of $\pts$.

It remains to prove that $\en(\pts) = \Omega(M(n)n^2)$.
For every integer $M(n)/32 \leq j \leq M(n)/16$, we denote by $\ell_j$ the vertical line $x = 4j$.
For every such $j$, there are $\Theta(n)$ points of $\pts$ that are to the left $\ell_j$, and the reflection of each such point across $\ell_j$ is another point of $\pts$.
That is, for every $M(n)/32 \leq j \leq M(n)/16$, the line $\ell_j$ is the perpendicular bisector of $\Theta(n)$ pairs of points of $\pts$.
The assertion of the theorem follows, since there are $\Theta(M(n))$ such lines, each contributing $\Theta(n^2)$ to $\en(\pts)$.


\section{Proof of Theorem \ref{th:fewdistances}}\label{sec:fewdistances}

In this section we prove that Theorem \ref{th:fewdistances} follows from Theorem \ref{th:energy}.
In fact, we prove the following more general version of Theorem \ref{th:fewdistances}.

\begin{theorem}\label{th:fewdistances2} Let $K(n)$ and $M(n)$ be two functions satisfying $ K(n)=O(\log n)$ and $ M(n)=O(n^{1/4})$.
If an $n$-point set $\pts\subset \RR^2$ spans $D=O(n/K(n))$ distinct distances,
then at least one of the following holds.\\
(i) There exists a line or a circle containing $M(n)$ points of $\pts$. \\
(ii) There are $\Omega(M(n)^{-\frac{12}{5}}n^{\frac{8}{5}-\eps})$
lines that contain $\Omega(K(n))$ points of $\pts$.
\end{theorem}

\noindent
Since Guth and Katz \cite{GK11} proved that any $n$-point set spans $\Omega(n/\log n)$ distinct distances, the assumption that $K=O(\log n)$ is not a real restriction.
The original formulation of Theorem \ref{th:fewdistances} is immediately obtained by setting $K(n)=\sqrt{\log n}$ and $M(n) = n^\alpha$.

\noindent \begin{proof}
For simplicity, we use the notation $K=K(n)$ and $M=M(n)$ throughout this proof.
We assume that (i) does not hold, and prove that (ii) holds in this case.

Given a point set $\pts\subset \RR^2$, we denote by $\bism(\pts)$ the \emph{multiset} of bisectors that are spanned by ordered pairs of $\ptss$.
Recall that $\bis(\pts)$ is the \emph{set} of distinct lines of $\bism(\pts)$.
For every line $\ell\in\bis(\pts)$, we denote by $\mu(\ell)$ its \emph{multiplicity} in $\bism(\pts)$ (i.e., the number of times it occurs in the multiset),
and set $\rho(\ell)=|\ell\cap \pts|$.
We define
\[I(\pts,\bism(\pts)) = \sum_{\ell\in \bis(\pts)} \mu(\ell)\rho(\ell);\]
that is, $I(\pts,\bism(\pts))$ is the number incidences with respect to their multiplicities.

We derive a lower bound on $I(\pts,\bism(\pts))$ by using an argument that is similar to the one in Szemer\'edi's proof of Theorem \ref{th:Szem}.
Let $T \subset \pts^3$ be the set of triples $(p,q,r)$ of distinct points of $\pts$ such that $|pq|=|pr|$.
Note that a triple $(p,q,r)$ is in $T$ if and only if $p$ is incident to $\bi{q,r}$. That is,
\[I(\pts,\bism(\pts)) = |T|.\]

Denote the distances that are determined by pairs of $\ptss$ as $\delta_1,\ldots,\delta_D$.
For every point $p\in\pts$ and $1\le i \le D$, let $\Delta_{i,p}$ denote the number of points of $\pts$ that have distance $\delta_i$ from $p$.
Let $T_p \subset T$ denote the set of triples of $T$ in which the first element is $p$.
Applying the Cauchy-Schwarz inequality yields
\[|T_p| = \Omega\left(\sum_{i=1}^D \Delta_{i,p}^2\right)
=\Omega\left(\frac{1}{D}\left(\sum_{i=1}^D \Delta_{i,p}\right)^2 \right)
= \Omega\left(\frac{n^2}{D}\right).\]
This in turn implies
\begin{equation} \label{eq:lowerT}
I(\pts,\bism(\pts)) = |T| = \sum_{p\in \pts} |T_p| = \Omega\left(\frac{n^3}{D}\right).
\end{equation}

We remark that by the Szemer\'edi-Trotter theorem \cite{ST83}, the number of incidences between $n$ points and $n^2$ distinct lines is $O\left(n^2\right)$.
This does not contradict \eqref{eq:lowerT} since the lines in the multiset $\bism(\pts)$ need not be distinct. Theoretically, it might be that $\bism(\pts)$ consists of $\Theta(n)$ distinct lines, each with multiplicity $\Theta(n)$ and incident to $\Theta(K)$ points.
However, our bound on the bisector energy excludes such cases.

Let $c_t$ be the constant implicit in the lower bound on $|T|$; we have
\[|T| \geq c_t Kn^2.\]
Let $L^+$ be the subset of lines in $\bis(\pts)$ that are each incident to at least $c_t K / 2$ points.
Then
\begin{align*}
c_t Kn^2 & \leq  \sum_{\ell\in \bis(\pts)} \mu(\ell)\rho(\ell), \\
&= \sum_{\ell \in L^+} \mu(\ell)\rho(\ell) + \sum_{\ell \in \bis(\pts) \minus L^+} \mu(\ell)\rho(\ell), \\
&\leq  \sum_{\ell \in L^+} \mu(\ell)\rho(\ell) + c_t K n^2 /2,
\end{align*}
where we use the fact that each ordered pair of points has a unique bisector, and hence contributes to $\sum_{\ell \in \bis(\pts)} \mu(\ell)$ exactly once.
Applying Cauchy-Schwarz, we get
\[
c_t^2 K^2 n^4/4 \leq \sum_{\ell \in L^+} \mu(\ell)^2 \sum_{\ell \in L^+} \rho(\ell)^2.
\]

Note that $\sum_{\ell \in \bis(\pts)} \mu(\ell)^2 = \Theta(\en(\pts))$. 
Since $M = O(n^{1/2}) = O(n^{2/3-\eps})$, Theorem \ref{th:energy} implies $\sum_{\ell \in \bis(\pts)} \mu(\ell)^2 = O(M^{2/5}n^{12/5+\eps})$.
We can bound $\sum \rho(\ell)^2$ using the assumption that no line contains more than $M$ points, so
\[ K^2 n^4 = O(M^{2/5}n^{12/5+\eps}\cdot M^2 |L^+|),\]
and hence
\[|L^+| = \Omega(K^2 n^{8/5 - \eps} M^{-12/5}).\]
Since $K = O(\log(n))$, it can be absorbed into the factor $n^\eps$ in the final bound.
\end{proof}

\paragraph{Remark.} Notice that the proof of Theorem \ref{th:fewdistances2} also applies when $M(n)= \Omega(n^{1/4})$. However, this would lead to a bound for the number of lines in \emph{(ii)} that is weaker than the bound that is implied by Theorem \ref{th:Szem}.


\section{Proof of Theorem \ref{th:incidencebound}}\label{sec:incidencebound}
We now present the proof of the incidence bound that we use. As mentioned in the introduction, this proof is essentially from \cite{FPSSZ14}; we reproduce it here to determine the dependence on the parameter $t$.
We prove a more general version than we need, since it seems to come at no extra cost, and may be useful elsewhere.

The proof uses the K\H ov\'ari-S\'os-Tur\'an theorem (e.g., see Bollob\'as \cite[Theorem IV.9]{Bol}),
which we formulate as a weak incidence bound.

\begin{lemma}[K\H ov\'ari-S\'os-Tur\'an]\label{le:weak}
Let $\vrts$ be a set of $n$ varieties and $\pts$ a set of $m$ points, both in $\RR^d$,
such that the incidence graph of $\pts\times\vrts$ contains no $K_{s,t}$ (where $s$ is constant but $t$ may depend on $m,n$).
Then
\[I(\pts,\vrts) = O(t^{1/s}mn^{(s-1)/s}+n).\]
\end{lemma}
We amplify the weak bound of Lemma \ref{le:weak} by using \emph{polynomial partitioning}.
Given a polynomial $f \in \RR[x_1,\ldots, x_d]$,
we write $Z(f) = \{ p \in \RR^d \mid f(p)=0 \}$.
We say that $f \in \RR[x_1,\ldots, x_d]$ is an \emph{$r$-partitioning polynomial} for a finite set $\pts\subset \RR^d$ if no connected component of $\RR^d \backslash Z(f)$ contains more than $|\pts|/r$ points of $\pts$ (notice that there is no restriction on the number of points of $\pts$ that are in $Z(f)$).
Guth and Katz \cite{GK11} introduced this notion and proved that for every $\pts\subset \RR^d$ and $1\le r \le |\pts|$, there exists an $r$-partitioning polynomial of degree $O(r^{1/d})$.
In \cite{FPSSZ14}, the following generalization was proved.

\begin{theorem}[Partitioning on a variety]\label{th:partitionV}
Let $V$ be an irreducible variety in $\RR^d$ of dimension $e$ and degree $D$.
Then for every finite $\pts\subset V$ there exists an $r$-partitioning polynomial $f$ of degree $O(r^{1/e})$ such that $V\not\subset Z(f)$.
The implicit constant depends only on $d$ and $D$.
\end{theorem}

We are now ready to prove our incidence bound. For the convenience of the reader, we first repeat the statement of the theorem.

\paragraph{Theorem \ref{th:incidencebound}.}
\emph{Let $\vrts$ be a set of $n$ constant-degree varieties and let $\pts$ be a set of $m$ points, both in $\RR^d$, such that the incidence graph of $\pts\times\vrts$ contains no copy of $K_{s,t}$ (where $s$ is a constant, but $t$ may depend on $m,n$).
Moreover, let $\pts\subset V$, where $V$ is an irreducible constant-degree variety of dimension $e$.
Then}
\[I(\pts,\vrts) = O\left(m^{\frac{s(e-1)}{es-1}+\eps}n^{\frac{e(s-1)}{es-1}}t^{\frac{e-1}{es-1}} +tm+n\right).\]

\begin{proof}
We use induction on $e$ and $m$, with the induction claim
\begin{equation} \label{eq:incK}
I(\pts,\vrts) \le \alpha_{1,e} m^{\frac{s(e-1)}{es-1}+\eps}n^{\frac{e(s-1)}{es-1}}t^{\frac{e-1}{es-1}}+\alpha_{2,e}(tm+n).
\end{equation}
The base cases for the induction are simple.
If $m$ is sufficiently small, then \eqref{eq:incK} follows immediately by choosing sufficiently large values for $\alpha_{1,e}$ and $\alpha_{2,e}$.
Similarly, when $e=0$, we again obtain \eqref{eq:incK} when $\alpha_{1,e}$ and $\alpha_{2,e}$ are sufficiently large (with respect to $d$ and to the degree of $V$).

The constants $d,e,D,s,1/\eps$ are given and thus fixed.
The other constants are to be chosen, and the dependencies between them are
\[ C_\text{weak}, \cp,\inter
\ll \cells
\ll \chol
\ll r
\ll \comps,\alpha_{1,e-1},\alpha_{2,e-1}
\ll \alpha_{2,e}
\ll \alpha_{1,e},\]
where $C\ll C'$ means that $C'$ is to be chosen sufficiently large compared to $C$; that is, $C$ should be chosen before $C'$.

By Lemma \ref{le:weak}, there exists a constant $C_\text{weak}$ (depending on $d,s$) such that
\[I(\pts,\vrts)\le C_\text{weak}\left(mn^{1-1/s}t^{1/s}+n\right).\]
 When $m\le (n/t)^{1/s}$, and $\alpha_{2,e}$ is sufficiently large, we have $I(\pts,\vrts) \le \alpha_{2,e} n$.
Therefore, in the remainder of the proof we can assume that $n < m^st$, which implies
\begin{equation} \label{eq:nElim}
n = n^{\frac{d-1}{ds-1}} n^{\frac{d(s-1)}{ds-1}} \le m^{\frac{s(d-1)}{ds-1}} n^{\frac{d(s-1)}{ds-1}}t^{\frac{(d-1)}{ds-1}}.
\end{equation}

\paragraph{Partitioning.}
By Theorem \ref{th:partitionV}, there exists an $r$-partitioning polynomial $f$ with respect to $V$ of degree at most $\cp \cdot r^{1/e}$, for a constant $\cp$.
Denote the cells of $V\backslash Z(f)$ as $\Omega_1, \ldots, \Omega_N$. Since we are working over the reals, there exists a constant-degree polynomial $g$ such that $Z(g)=V$. Then, by \cite[Theorem A.2]{ST11}, we have
$N\le \cells \cdot \deg(f)^{\dim V}=\cells \cdot r$ for some constant $\cells$ depending on $\cp$.

We partition $I(\pts,\vrts)$ into the following three subsets:

\begin{itemize}
\item $I_1$ consists of the incidences $(p,S) \in \pts\times\vrts$ such that $p\in V \cap Z(f)$,
and some irreducible component of $V \cap Z(f)$ contains $p$ and is fully contained in $S$.
\item $I_2$ consists of the incidences $(p,S) \in \pts\times\vrts$ such that $p\in V \cap Z(f)$,
and every irreducible component of $V\cap Z(f)$ that contains $p$ is not contained in $S$.
\item $I_3 = I(\pts,\vrts) \backslash (I_1\cup I_2)$, the set of incidences $(p,S) \in \pts\times\vrts$ such that $p$ is not contained in $V \cap Z(f)$.
\end{itemize}

\noindent Note that we indeed have $I(\pts,\vrts) = I_1 + I_2+I_3$.

\paragraph{Bounding $I_1$.}
The points of $\pts\subset \RR^d$ that participate in incidences of $I_1$ are all contained in the variety $V_0 =V \cap Z(f)$. Set $\pts_0=\pts \cap V_0$ and $m_0 = |\pts_0|$.
Since $V$ is an irreducible variety and $V\not\subset Z(f)$, $V_0$ is a variety of dimension at most $e-1$ and of degree that depends on $r$.
By \cite[Lemma 4.3]{ST11}, the intersection $V_0$ is a union of $\comps$ irreducible components, where $\comps$ is a constant depending on $r$ and $d$.\footnote{This lemma only applies to complex varieties. However, we can take the \emph{complexification} of the real variety and apply the lemma to it (for the definition of a complexification, e.g., see \cite[Section 10]{Whit57}). The number of irreducible components of the complexification cannot be smaller than number of irreducible components of the real variety (e.g., see \cite[Lemma 7]{Whit57}).} The degrees of these components also depend only on these values (for a proper definition of degrees and further discussion, e.g., see \cite{FPSSZ14}).

Consider an irreducible component $W$ of $V_0$. If $W$ contains at most $s-1$ points of $\pts_0$, it yields at most $(s-1)n$ incidences. Otherwise, since the incidence graph contains no $K_{s,t}$, there are at most $t-1$ varieties of $\vrts$ that fully contain $W$, yielding at most $(t-1)m_0$ incidences. By summing up, choosing sufficiently large $\alpha_{1,e},\alpha_{2,e}$, and applying \eqref{eq:nElim}, we have
\begin{equation} \label{eq:I_1}
I_1 \le \comps \left( sn + tm_0 \right)
< \frac{\alpha_{2,e}}{2}(n+tm_0)
< \frac{\alpha_{1,e}}{4}m^{\frac{s(e-1)}{es-1}} n^{\frac{e(s-1)}{es-1}}t^{\frac{(e-1)}{es-1}} + \frac{\alpha_{2,e}}{2}tm_0.
\end{equation}

\paragraph{Bounding $I_2$.}
The points that participate in $I_2$ lie in $V_0=V\cap Z(f)$,
and the varieties that participate do not contain any component of $V_0$.
Because $V_0$ has dimension at most $e-1$, and the participating varieties do not contain any component of $V_0$, we can apply the induction claim on each irreducible component of $V_0$.
Since $V_0$ has $\comps$ irreducible components, we get
\begin{equation*}
I_2 \le \comps \alpha_{1,e-1} m_0^{\frac{s(e-2)}{(e-1)s-1}+\eps}n^{\frac{(e-1)(s-1)}{(e-1)s-1}}t^{\frac{e-2}{(e-1)s-1}}+ \alpha_{2,e-1}(tm_0+n),
\end{equation*}
with $\alpha_{1,e-1}$ and $\alpha_{2,e-1}$ depending on the degree of the irreducible component of $V_0$, which in turn depends on $r$.
The analysis that leads to \eqref{eq:nElim} also yields the following bound.
\[ m^{\frac{s(e-2)}{(e-1)s-1}+\eps}n^{\frac{(e-1)(s-1)}{(e-1)s-1}}t^{\frac{e-2}{(e-1)s-1}} \le m^{\frac{s(e-1)}{es-1}+\eps}n^{\frac{e(s-1)}{es-1}}t^{\frac{e-1}{es-1}}.\]
By applying \eqref{eq:nElim} to remove the term $\alpha_{2,e-1} n$, and by choosing $\alpha_{1,e}$ and $\alpha_{2,e}$ sufficiently large with respect to $\comps,\alpha_{1,e-1},\alpha_{2,e-1}$,  we obtain
\begin{equation}\label{eq:I_2}
I_2 \le \frac{\alpha_{1,e}}{4} m^{\frac{s(e-1)}{es-1}+\eps}n^{\frac{e(s-1)}{es-1}}t^{\frac{e-1}{es-1}}+\frac{\alpha_{2,e}}{2}tm_0.
\end{equation}

\paragraph{Bounding $I_3$.}
For every $1 \le i \le N$, we set $\pts_i = \pts\cap \Omega_i$
and denote by $\vrts_i$ the set of varieties of $\vrts$ that intersect the cell $\Omega_i$ but do not contain it. We also set $m_i=|P_i|$ and $n_i=|\vrts_i|$. Since $f$ is an $r$-partitioning polynomial, we have $m_i \le m/r$, .

We have $\sum_{i=1}^N m_i = m-m_0$.
By \cite[Theorem A.2]{ST11}, there exists a constant $\inter$ such that the following holds for every $\vrt\in \vrts$.
The subvariety $\vrt\cap V$ of $V$, which must have dimension at most $e-1$, intersects at most $\inter \cdot\deg(f)^{\dim(S\cap V)} = \inter \cdot r^{(e-1)/e}$ cells.
This implies that
\begin{equation*}
\sum_{i=1}^N n_i \le \inter \cdot r^{(e-1)/e}\cdot n.
\end{equation*}
By H\"older's inequality we have
\begin{align*}
\sum_{i=1}^N n_i^{\frac{e(s-1)}{es-1}} &
\le \left(\sum_{i=1}^N n_i\right)^{\frac{e(s-1)}{es-1}} \left(\sum_{i=1}^N 1\right)^{\frac{e-1}{es-1}} \\
&\le \left(\inter  r^{(e-1)/e} n\right)^{\frac{e(s-1)}{es-1}}\left(\cells  r\right)^{\frac{e-1}{es-1}} \\
&\le \chol r^{\frac{(e-1)s}{es-1}} n^{\frac{e(s-1)}{es-1}},
\end{align*}
where $\chol$ depends on $\inter,\cells$.
Using the induction hypothesis, we obtain
\begin{align*}
\sum_{i=1}^N I(\pts_i,\vrts_i) &\le \sum_{i=1}^N \left(\alpha_{1,e} m_i^{\frac{(e-1)s}{es-1}+\varepsilon}n_i^{\frac{e(s-1)}{es-1}}t^{\frac{(e-1)}{es-1}}+\alpha_{2,e}(tm_i+n_i)\right) \\
&\le \alpha_{1,e} \frac{m^{\frac{(e-1)s}{es-1}+\varepsilon}t^{\frac{(e-1)}{es-1}}}{r^{\frac{(e-1)s}{es-1}+\varepsilon}} \sum_{i=1}^N n_i^{\frac{e(s-1)}{es-1}} + \sum_{i=1}^N\alpha_{2,e}(tm_i+n_i) \\[2mm]
&\le \alpha_{1,e} \chol \frac{m^{\frac{(e-1)s}{es-1}+\varepsilon}n^{\frac{e(s-1)}{es-1}}t^{\frac{(e-1)}{es-1}}}{r^{\varepsilon}}   + \alpha_{2,e}\left(t(m-m_0)+ \inter r^{\frac{e-1}{e}}n\right).
\end{align*}
By choosing $\alpha_{1,e}$ sufficiently large with respect to $\inter,r,\alpha_{2,e}$, and using \eqref{eq:nElim}, we get
\[ \sum_{i=1}^N I(\pts_i,\vrts_i) \le 2\alpha_{1,e} \chol \frac{m^{\frac{(e-1)s}{es-1}+\varepsilon}n^{\frac{e(s-1)}{es-1}}s^{\frac{(e-1)}{es-1}}}{r^{\varepsilon}} + \alpha_{2,e}t(m-m_0). \]
Finally, choosing $r$ sufficiently large with respect to $\chol$ gives
\begin{equation}\label{eq:I_3}
I_3 = \sum_{i=1}^N I(\pts_i,\vrts_i) \le \frac{\alpha_{1,e}}{2} m^{\frac{(e-1)s}{es-1}+\varepsilon}n^{\frac{e(s-1)}{es-1}}t^{\frac{(e-1)}{es-1}} + \alpha_{2,e}t(m-m_0).
\end{equation}

\paragraph{Summing up.} By combining $I(\pts,\vrts) = I_1 + I_2+I_3$ with \eqref{eq:I_1}, \eqref{eq:I_2}, and \eqref{eq:I_3}, we obtain
\[ I(\pts,\vrts) \le \alpha_{1,e} m^{\frac{s(e-1)}{es-1}+\eps}n^{\frac{e(s-1)}{es-1}}t^{\frac{(e-1)}{es-1}}+\alpha_{2,e}(tm+n), \]
which completes the induction step and the proof of the theorem.
\end{proof}







\end{document}